\pdfoutput=1
\RequirePackage{ifpdf}
\ifpdf 
\documentclass[pdftex]{sigma}
\else
\documentclass{sigma}
\fi

\newcommand{\C}{\mathbb{C}}

\newcommand{\Z}{\mathbb{Z}}

\newcommand{\N}{\mathbb{N}}
\newcommand{\R}{\mathbb{R}}

\newcommand{\re}{\textnormal{Re}}

\numberwithin{equation}{section}

\newtheorem{Theorem}{Theorem}[section]
\newtheorem*{Theorem*}{Theorem}
\newtheorem{Corollary}[Theorem]{Corollary}
\newtheorem{Lemma}[Theorem]{Lemma}
\newtheorem{Proposition}[Theorem]{Proposition}
{ \theoremstyle{definition}

	\newtheorem{Remark}[Theorem]{Remark} }

\begin{document}

\allowdisplaybreaks

\renewcommand{\thefootnote}{}

\newcommand{\arXivNumber}{2307.12955}

\renewcommand{\PaperNumber}{001}

\FirstPageHeading

\ShortArticleName{A Note on the Equidistribution of 3-Colour Partitions}

\ArticleName{A Note on the Equidistribution of 3-Colour Partitions\footnote{This paper is a~contribution to the Special Issue on Asymptotics and Applications of Special Functions in Memory of Richard Paris. The~full collection is available at \href{https://www.emis.de/journals/SIGMA/Paris.html}{https://www.emis.de/journals/SIGMA/Paris.html}}}

\Author{Joshua MALES~$^{\rm ab}$}

\AuthorNameForHeading{J.~Males}

\Address{$^{\rm a)}$~School of Mathematics, University of Bristol, Bristol, BS8 1TW, UK}
\EmailD{\href{mailto:joshua.males@bristol.ac.uk}{joshua.males@bristol.ac.uk}}
\URLaddressD{\url{http://www.joshuamales.com}}

\Address{$^{\rm b)}$~Heilbronn Institute for Mathematical Research, University of Bristol, Bristol, BS8~1UG, UK}

\ArticleDates{Received July 25, 2023, in final form December 28, 2023; Published online January 01, 2024}

\Abstract{In this short note, we prove equidistribution results regarding three families of three-colour partitions recently introduced by Schlosser and Zhou. To do so, we prove an asymptotic formula for the infinite product $F_{a,c}(\zeta ; {\rm e}^{-z}) := \prod_{n \geq 0} \big(1- \zeta {\rm e}^{-(a+cn)z}\big)$ ($a,c \in \N$ with $0<a\leq c$ and $\zeta$ a root of unity) when $z$ lies in certain sectors in the right half-plane, which may be useful in studying similar problems. As a corollary, we obtain the asymptotic behaviour of the three-colour partition families at hand.}


\Classification{11P82}

\renewcommand{\thefootnote}{\arabic{footnote}}
\setcounter{footnote}{0}

\section{Introduction}
The study of partitions has a long and storied history, and has been a fruitful field for mathematicians over the last century (and beyond). By a partition of a positive integer~$n$, we mean a non-increasing list of positive integers $\lambda_\ell$ such that $\sum_\ell \lambda_\ell =n$. Letting $p(n)$ count the total number of partitions of $n$, a famous result of Hardy and Ramanujan \cite{HardyRamanujan} is their celebrated asymptotic
\begin{equation*}
	p(n)\sim \frac{1}{4\sqrt{3}n} {\rm e}^{\pi \sqrt{\frac{2n}{3}}},
\end{equation*}
as $n\rightarrow \infty$. In order to prove their result, they introduced the circle method~-- a powerful tool used widely in number theory which has various incarnations in the literature.

Over the next 100 years, mathematicians have studied numerous problems in the partition-theoretic world using the circle method. In recent years, one such problem has been the equidistribution properties of partition-theoretic objects over arithmetic progressions. For example, in~\cite{BCMO} Bringmann, Craig, Ono, and the author studied the asymptotic equidistribution of Betti numbers of certain Hilbert schemes, as well as the non-equidistribution of $t$-hooks of partitions.

In this short note, we prove further equidistribution results for three new families of three-colour partitions that were very recently introduced by Schlosser and Zhou~\cite{SZ}. In order to prove our results, we require an asymptotic for the function
\begin{align*}
	F_{a,c}(\zeta ; q) := \prod_{n \geq 0} \big(1- \zeta q^{a+cn}\big),
\end{align*}
where $a,c \in \N$ with $0<a\leq c$ and $\zeta$ a root of unity. We remark that $F_{1,1}$ is one of the functions studied in \cite{BCMO}, while $F_{a,c}(1;q)$ is one of the central objects studied by Craig in \cite{C}, and so this note lies at the intersection of these two papers. We note that there are many related studies in the area, for example the paper \cite{Ciolan} by Ciolan which discusses the parity of partitions in $k$-th powers and the paper \cite{Zhou} by Zhou which proves similar equidistribution properties for partitions with parts taken from a sequences determined by polynomials satisfying certain conditions.

In \cite{SZ}, the authors studied certain basic hypergeometric functions and their truncations. In turn, this allowed them to prove inequalities on certain families of three colour partitions.

Firstly, let $J_E(m,n)$ count the number of partitions of $n$ for which the difference between the number of parts congruent to $1 \pmod{3}$ and the number of parts congruent to $2 \pmod{3}$ is equal to $m$. Secondly, let $J_T(m,n)$ count number of partitions of $n$ into, say, red, green and blue parts, for which the difference between the numbers of red and green parts is equal to $m$. Finally, $J_G(m,n)$ count the number of partitions of $n$ into three colours, say red, green and blue, such that the size of each of the red and green parts are odd, the size of blue parts is even, and the difference between the numbers of red and green parts is equal to $m$.

Then in \cite{SZ} the authors showed that
\begin{align*}
	& J_E(\zeta;q) := \sum_{n \geq 0} \sum_{m \in \Z} J_E(m,n) \zeta^m q^n= \frac{1}{\big( \zeta q, \zeta^{-1}q^2, q^3;q^3\big)_\infty}, \\
		&J_T(\zeta;q) := \sum_{n \geq 0} \sum_{m \in \Z} J_T(m,n) \zeta^m q^n= \frac{1}{\big( \zeta q, \zeta^{-1}q, q;q\big)_\infty}, \\
	&J_G(\zeta;q) := \sum_{n \geq 0} \sum_{m \in \Z} J_G(m,n) \zeta^m q^n = \frac{1}{\big(\zeta q,\zeta^{-1}q,q^2;q^2\big)_\infty},
\end{align*}
where $(a;q)_\infty$ is the usual $q$-Pochhammer symbol and we define \[(a,b;q)_\infty := (a;q)_\infty (b;q)_\infty.\]

For $* \in \{E,T,G\}$, let $J_*(r,s;n)$ count the corresponding partition family but restricted to count when the given statistic is congruent to $r$ modulo $s$. Then the main purpose of this note is to prove the following theorem.

\begin{Theorem}\label{thm: main}
	The following equidistribution results are true.
	\begin{enumerate}
		\item[$(1)$] Let $s \in \N$ such that $3 \nmid s$. Then
		\begin{align*}
			J_E(r,s;n) = \frac{1}{4 \sqrt{3} s n}{\rm e}^{\pi \sqrt{\frac{2n}{3}}} \big(1+O\big(n^{-\frac{1}{2}}\big)\big).
		\end{align*}
		
		\item[$(2)$] Let $s \in \N$. Then
		\begin{align*}
			J_T(r,s;n) = \frac{1}{2^{\frac{7}{2}} s n^{\frac{3}{2}}} {\rm e}^{\pi\sqrt{2n}} \big(1+O\big(n^{-\frac{1}{2}}\big)\big).
		\end{align*}
		
		\item[$(3)$] Let $s\in \N$ such that $2 \nmid s$. Then
		\begin{align*}
			J_G(r,s;n) = \frac{{\rm e}^{\pi \sqrt{n}}}{8sn} \big(1+O\big(n^{-\frac{1}{2}}\big)\big).
		\end{align*}
	\end{enumerate}
\end{Theorem}

\begin{Remark}
	The conditions on $s$ in the first and third cases ensure that a single term in the asymptotic behaviour dominates when we take a multisection of the appropriate generating function. It is noted in the course of the proof how these restrictions could be lifted.
\end{Remark}

Let $J_*(n) := \sum_{m \in \Z} J_*(m,n)$. As an immediate corollary, we obtain the asymptotic behaviour of $J_E(n)$, $J_T(n)$ and $J_G(n)$.
\begin{Corollary}
	We have the following asymptotics.
	\begin{enumerate}
		\item[$(1)$] As $n \to \infty$
		\begin{align*}
			J_E(n) = \frac{1}{4 \sqrt{3} n}{\rm e}^{\pi \sqrt{\frac{2n}{3}}} \big(1+O\big(n^{-\frac{1}{2}}\big)\big).
		\end{align*}
		
		\item[$(2)$] As $n \to \infty$
		\begin{align*}
			J_T(n) = \frac{1}{2^{\frac{7}{2}} n^{\frac{3}{2}}} {\rm e}^{\pi\sqrt{2n}} \big(1+O\big(n^{-\frac{1}{2}}\big)\big).
		\end{align*}
		
		\item[$(3)$] As $n \to \infty$
		\begin{align*}
			J_G(n) = \frac{{\rm e}^{\pi \sqrt{n}}}{8n} \big(1+O\big(n^{-\frac{1}{2}}\big)\big).
		\end{align*}
	\end{enumerate}
\end{Corollary}

\begin{Remark}
	In fact, we have that $J_E(n) = p(n)$ and $J_G(n) = \overline{p}(n)$, the number of overlined partitions of $n$, and we recover their classical asymptotic formulae here.
\end{Remark}

We remark that this is of course not the only approach to obtaining asymptotics of such generating functions, though we chose the presented approach to include the asymptotic results on the functions $F_{a,c}$. For example, work of Liu and Zhou \cite{LZ} also provides the asymptotic behaviour that is outlined here (and would also give a more precise error term). We leave the details for the interested reader.

\section{Preliminaries}

\subsection{Euler--Maclaurin summation}
We require the following generalization of the famous Euler--Maclaurin summation formula, proved in~\cite{BCMO}~-- we follow the set-up there. Consider a function $f$ on a domain in $\C$ which is of sufficient decay. That is, there exists $\varepsilon >0$ such that $f(w) \ll w^{-1-\varepsilon}$ as $|w| \rightarrow \infty$ in the domain.

 For $0\leq \theta < \frac{\pi}{2}$, define
\begin{align}\label{eqn: defn D_theta}
	D_{\theta} := \big\{ z=r{\rm e}^{{\rm i}\alpha} \colon r \geq 0 \text{ and } |\alpha| \leq \theta \big\}.
\end{align}
 For $s,z\in\C$ with $\operatorname{Re}(s)>1$, $\operatorname{Re}(z)>0$, the Hurwitz zeta function is defined by $\zeta(s,z):=\sum_{n=0}^\infty \frac{1}{(n+z)^s}$. We also need the digamma function $\psi(x):=\frac{\Gamma'(x)}{\Gamma(x)},$ and the Euler--Mascheroni constant $\gamma$. Let $B_n(x)$ be the $n$-th Bernoulli polynomial.

\begin{Lemma}[{\cite[Lemma 2.2]{BCMO}}]\label{lemma11}
	Let $0 < a \leq 1$ and $A \in \R^+$, and let $D_{\theta}$ be defined by \eqref{eqn: defn D_theta}. Assume that $f(z) \sim \sum_{n=n_0}^{\infty} c_n z^n$ $(n_0\in\Z)$ as $z \rightarrow 0$ in $D_\theta$. Furthermore, assume that $f$ and all of its derivatives are of sufficient decay in $D_\theta$. Then we have that
	\begin{align*}
		\sum_{n=0}^\infty f((n+a)z)\sim{} &\sum_{n=n_0}^{-2} c_{n} \zeta(-n,a)z^{n}+ \frac{I_{f,A}^*}{z}-\frac{c_{-1}}{z} ( \log (Az) +\psi(a)+\gamma )\\
&-\sum_{n=0}^\infty c_n \frac{B_{n+1}(a)}{n+1} z^n,
	\end{align*}
	as $z \rightarrow 0$ uniformly in $D_\theta$, where
	\begin{align*}
		I_{f,A}^*:=\int_{0}^{\infty} \left(f(u)-\sum_{n=n_0}^{-2}c_{n}u^n-\frac{c_{-1}{\rm e}^{-Au}}{u}\right){\rm d}u.
	\end{align*}
\end{Lemma}
\begin{Remark}
	When $a=1$, we have that $\psi(a)+\gamma=0$.
\end{Remark}

\subsection{Wright's circle method}
In order to obtain the asymptotic behaviour of our coefficients, we make use of Wright's circle method, which is a more flexible version of the original circle method of Hardy and Ramanujan developed by Wright in \cite{Wright1,Wright2}. In essence, one uses Cauchy's residue theorem to write the coefficients as a contour integral of the generating function over a contour (chosen to be a~circle)~$C$ of radius less than one. Since we may choose the radius of the circle, we pick a radius such that $C$ tends to the unit circle as $n\to\infty$. This ensures that we may use the asymptotics given by Lemma~\ref{lemma11} for our generating function toward roots of unity on the unit circle. One then splits the contour into arcs where the generating functions has relatively large (resp.\ small) asymptotic growth, called the major (resp.\ minor) arcs. On the major arcs, we use asymptotic techniques to closely approximate the behaviour of the generating function, while on the minor arcs we bound more crudely.

In \cite{BCMO}, following work of Ngo and Rhoades \cite{NgoRhoades}, the following result was proved.

\begin{Proposition}[{\cite[Proposition 4.4]{BCMO}}] \label{WrightCircleMethod}
	Suppose that $F(q)$ is analytic for $q = {\rm e}^{-z}$, where ${z=x+{\rm i}y \in \C}$ satisfies $x > 0$ and $|y| < \pi$, and suppose that $F(q)$ has an expansion \[F(q) = \sum_{n=0}^\infty c(n) q^n\] near $1$. Let $c,N,M>0$ be fixed constants. Consider the following hypotheses:
	\begin{enumerate}
		\item[\rm(1)] As $z\to 0$ in the bounded cone $|y|\le Mx$ $($major arc$)$, we have
		\begin{align*}
			F({\rm e}^{-z}) = z^{B} {\rm e}^{\frac{A}{z}} \left( \sum_{j=0}^{N-1} \alpha_j z^j + O\big(|z|^N\big) \right),
		\end{align*}
		where $\alpha_s \in \C$, $A\in \R^+$, and $B \in \R$.
		
		\item[\rm(2)] In the bounded sector $Mx\le|y| < \pi$ $($minor arc$)$, we have
		\begin{align*}
			\lvert	F({\rm e}^{-z}) \rvert \ll {\rm e}^{\frac{1}{\mathrm{Re}(z)}(A - \kappa)}.
		\end{align*}
		for some $\kappa\in \R^+$.
	\end{enumerate}
	If {\rm(1)} and {\rm(2)} hold, then as $n \to \infty$, we have for any $N\in \R^+$
	\begin{align*}
		c(n) = n^{\frac{1}{4}(- 2B -3)}{\rm e}^{2\sqrt{An}} \left( \sum\limits_{r=0}^{N-1} p_r n^{-\frac{r}{2}} + O\big(n^{-\frac N2}\big) \right),
	\end{align*}
	where $p_r := \sum\limits_{j=0}^r \alpha_j c_{j,r-j}$ and
	\begin{align*}
		c_{j,r} := \dfrac{\bigl(-\frac{1}{4\sqrt{A}}\bigr)^r \sqrt{A}^{j + B + \frac 12}}{2\sqrt{\pi}} \dfrac{\Gamma\big(j + B + \frac 32 + r\big)}{r! \Gamma\big(j + B + \frac 32 - r\big)}.
	\end{align*}
\end{Proposition}

This result means that one need only verify the two hypotheses in order to obtain the asymptotic behaviour of the coefficients.

\subsection[Asymptotics of (zeta q\^{}a;q\^{}c)\_infty]{Asymptotics of $\boldsymbol{(\zeta q^a;q^c)_\infty}$}

For $0 < a \leq c$, consider the infinite product
\begin{align*}
	F_{a,c}(\zeta ; q) := \prod_{n \geq 0} \left(1- \zeta q^{a+cn}\right) = (\zeta q^a;q^c)_\infty.
\end{align*}
Notice that
\begin{align*}
	F_{1,1} (\zeta;q) = \prod_{n \geq 1} \left(1- \zeta q^{n}\right)
\end{align*}
is a function studied by Bringmann, Craig, Ono, and the author in \cite{BCMO}. The following proposition gives the asymptotics of $F_{a,c}(\zeta ; {\rm e}^{-z})$ where $z \to 0$ in $D_\theta$, mirroring \cite[Theorem~2.1]{BCMO}. To state the result, recall Lerch's transcendent
\begin{align*}
	\Phi(z,s,a):=\sum_{n=0}^\infty \frac{z^n}{(n+a)^s}.
\end{align*}

\begin{Proposition}\label{Prop: F_a,c}
	Let $\zeta$ be a primitive $b$-th root of unity for $b \geq 2$. As $z\to 0$ in $D_\theta$, we have that
	\begin{align*}
		F_{a,c}(\zeta;{\rm e}^{-z}) = (1-\zeta)^{\frac{1}{2} - \frac{a}{c}}\exp\left( -\frac{\zeta \Phi (\zeta,2, 1 )}{ c z} \right) ( 1+ O(|z|) ).
	\end{align*}
\end{Proposition}
\begin{proof}
	First note that
	\begin{align*}
		\log F_{a,c}(\zeta;q) = \sum_{n \geq 0} \log\left(1- \zeta q^{a+cn}\right) &= -\sum_{n \geq 0} \sum_{k \geq 1} \frac{\zeta^k q^{ak+cnk}}{k} = -\sum_{k \geq 1} \frac{\zeta^k q^{ak}}{k \big(1-q^{ck}\big)}.
	\end{align*}
Let $k = mb+j$ to obtain
\begin{align*}
	- \sum_{j=1}^{b} \zeta^j \sum_{m \geq 0} \frac{ q^{a(mb+j)}}{(mb+j) \big(1-q^{c(mb+j)}\big)}.
\end{align*}
We may rewrite this as
\begin{align*}
	- z \sum_{j=1}^{b} \zeta^j \sum_{m \geq 0} f\left( \left(m +\frac{j}{b}\right) bz \right),
\end{align*}
where
\begin{align*}
	f(x) := \frac{{\rm e}^{-ax}}{x \left(1- {\rm e}^{-cx}\right)}.
\end{align*}
Note that $f$ has a Taylor expansion at $x=0$ of $\frac{1}{cx^2} + \frac{\frac{1}{2} - \frac{a}{c}}{x} +O(1)$.
Thus using Lemma \ref{lemma11}, we obtain that
\begin{align*}
		\log F_{a,c}(\zeta;q) = {}& \sum_{j=1}^{b} \zeta^j \biggl( -\frac{1}{b^2 c z} \zeta\left(2,\frac{j}{b} \right) - \frac{I^*_{f,1}}{b}\\
& + \frac{\frac{1}{2} -\frac{a}{c}}{b} \left(\log(bz) + \psi\left(\frac{j}{b}\right) +\gamma \right) +O(|z|) \biggr).
\end{align*}
Then using that $\sum_{j=1}^b \zeta^{j}=0$ and \cite[p.~39]{Campbell} (see \cite{BCMO} which corrects a minus sign and erroneous~$k$ on the right-hand side of Campbell)
\begin{align*}
	\sum_{j=1}^b \psi \left( \frac{j}{b} \right) \zeta^{j}=b \log ( 1-\zeta )
\end{align*}
along with
\begin{align*}
	\sum_{j =1}^{b} \zeta^j \zeta\left(2 , \frac{j}{b}\right) = b^2 \zeta \Phi(\zeta,2,1)
\end{align*}
and exponentiating gives us that
\begin{align*}
		F_{a,c}(\zeta;q) = (1-\zeta)^{\frac{1}{2} - \frac{a}{c}}\exp\left( - \frac{\zeta \Phi(\zeta,2,1)}{ c z} \right) ( 1+ O(|z|) )
\end{align*}
as desired.
\end{proof}

Now consider $F_{a,c}(1;q)$. This function was considered by Craig \cite{C} who proved very precise asymptotics with explicit error bounds. Since we do not need the error terms explicitly in this paper, we state a more relaxed version of Craig's results.
\begin{Proposition}[{\cite[Proposition~3.1]{C}}]\label{Prop: Craig}
	As $z \to 0$ in $D_\theta$, we have that
	\begin{align*}
		F_{a,c}(1;{\rm e}^{-z}) = \frac{\Gamma\left(\frac{a}{c}\right)}{\sqrt{2\pi}} (cz)^{\frac{1}{2}-\frac{a}{c}} \exp\left(\frac{\pi^2}{6 cz} \right) (1+O(|z|) ).
	\end{align*}
	
\end{Proposition}

\section[Proof of Theorem 1.1]{Proof of Theorem \ref{thm: main}}

First note that we have
\begin{align*}
	& J_E(\zeta;q) = F_{1,3}^{-1}(\zeta;q) F_{2,3}^{-1}\big(\zeta^{-1};q\big) \big(q^3;q^3\big)_\infty^{-1},\qquad J_T(\zeta;q) = F_{1,1}^{-1}(\zeta;q) F_{1,1}^{-1}\big(\zeta^{-1};q\big) (q;q)^{-1}_\infty,\\
	&J_G(\zeta;q) = F_{1,2}^{-1}(\zeta;q) F_{1,2}^{-1}\big(\zeta^{-1};q\big) \big(q^2;q^2\big)^{-1}_\infty.
\end{align*}

Recall the well-known bound (for $|y|\le Mx$, as $z\to0$)
\begin{align*}
	\left({\rm e}^{-z};{\rm e}^{-z}\right)_\infty^{-1} = \sqrt{\frac{z}{2\pi}} {\rm e}^{\frac{\pi^2}{6z}} (1+O(|z|)).
\end{align*}
By combining this with the asymptotics given in Proposition~\ref{Prop: F_a,c}, we immediately obtain the following corollary.
\begin{Corollary}
	Let $\zeta$ be a primitive $b$-th root of unity. Then as $z \to 0$ in $D_\theta$, we have that
	\begin{align*}
J_E(\zeta;{\rm e}^{-z}) ={}& (1-\zeta)^{- \frac{1}{6}} \big(1-\zeta^{-1}\big)^{- \frac{1}{6}} \\
&\times \sqrt{\frac{3z}{2\pi}} \exp\left( \frac{\pi^2}{18z} + \frac{\zeta \Phi(\zeta,2,1) + \zeta^{-1} \Phi\big(\zeta^{-1},2,1\big)}{ 3z} \right) \left( 1+O(|z|)\right), \\
J_T(\zeta;{\rm e}^{-z}) ={}& (1-\zeta)^{\frac{1}{2}} \big(1-\zeta^{-1}\big)^{\frac{1}{2}} \\
&\times\sqrt{\frac{z}{2\pi}} \exp\left( \frac{\pi^2}{6z} + \frac{\zeta \Phi(\zeta,2,1) + \zeta^{-1} \Phi\big(\zeta^{-1},2,1\big)}{ z} \right) ( 1+O(|z|)), \\
J_G(\zeta;{\rm e}^{-z}) ={}& \sqrt{\frac{z}{\pi}} \exp\left( \frac{\pi^2}{12z} + \frac{\zeta \Phi(\zeta,2,1) + \zeta^{-1} \Phi\big(\zeta^{-1},2,1\big)}{ 2z} \right) ( 1+O(|z|)).
	\end{align*}
\end{Corollary}

Using these results, we now prove Theorem \ref{thm: main}.

\begin{proof}[Proof of Theorem \ref{thm: main}]
It is a simple exercise using orthogonality of roots of unity to show that
\begin{align}\label{eqn: multisection}
H_*(r,s;q) := \sum_{n \geq 0} J_*(r,s;n) q^n = \frac{1}{s} \sum_{j=0}^{s-1} \zeta_b^{-rj} J_*\big(\zeta_s^j;q\big).
\end{align}

We now use Wright's circle method to show that the $j=0$ term dominates the asmyptotics for each choice $* \in \{ E,T,G \}$, which in turn implies that the main term asymptotic is independent of the residue class and hence the statistic is asymptotically equidistributed.

One may be concerned that poles other than $q=1$ may contribute to the main asymptotic term. However, by Theorem~1.1 of the beautiful paper of Chern~\cite{Chern} we know that for each of the infinite products, the pole at $q=1$ dominates, since each $F_{a,c}^{-1}$ is exponentially smaller at poles other than $q=1$. This is only true because of the assumptions on $s$ made in each case. For example, in case (3) if one had terms where $\zeta = -1$ then the pole at $q =-1$ would also have a term with the same growth as the main term asymptotic, and would cancel the main term already obtained thanks to the sum over $j$ in \eqref{eqn: multisection}. We leave the details of these further cases to the interested reader.

Using that $\lvert \zeta \Phi(\zeta,2,1) \rvert < \frac{\pi^2}{6}$ for $\zeta$ is a primitive root of unity not equal to~$1$, it is simple to see that on the major arc the $j=0$ term in \eqref{eqn: multisection} exponentially dominates all other terms, and thus using Proposition~\ref{Prop: Craig}, we have
\begin{align*}
		&H_E(r,s;q) = \frac{\Gamma\left(\frac{1}{3}\right) \Gamma\left(\frac{2}{3}\right)}{s (2\pi)^{\frac{3}{2}}} \sqrt{3z} \exp\left( \frac{\pi^2}{6z} \right) \left(1+O(|z|)\right) , \\
	&H_T(r,s;q) = \frac{1}{s (2\pi)^{\frac{3}{2}}} z^{\frac{3}{2}} \exp\left(\frac{\pi^2}{2z}\right) \left(1+O(|z|)\right), \\
	&H_G(r,s;q) = \frac{1}{2s} \sqrt{\frac{z}{\pi}} \exp\left( \frac{\pi^2}{4z} \right)\left(1+O(|z|)\right).
\end{align*}
Note that by Euler's reflection formula for the $\Gamma$-function, we obtain $\Gamma\left(\frac{1}{3}\right) \Gamma\left(\frac{2}{3}\right) = \frac{2 \pi}{\sqrt{3}}$.

 It remains to estimate each of $H_*(r,s;q)$ on the minor arcs, where $Mx \le |y|\le \pi$. To do so, we estimate $F_{a,c}(\zeta;q)$ on the minor arcs using arguments very similar to those given in \cite{BCMO} along with the classical fact that for some $\mathcal{C}>0$, we have
 \begin{align*}
 \big|({\rm e}^{-z};{\rm e}^{-z})_\infty^{-1}\big| \le x^{\frac12} {\rm e}^{\frac{\pi}{6x}-\frac{\mathcal{C}}{x}}.
 \end{align*}
This guarantees that the final $q$-Pochhammer in each case is exponentially smaller on the minor arc, and so we need only consider the $F_{a,c}$ functions.

We aim to show that (for $Mx \le |y|\le \pi$)
\begin{align*}
	\re\big(\log F_{a,c}^{-1}\big(\zeta_b^j ; q \big) \big) - \log\big(F_{a,c}^{-1}(1;|q|)\big) \leq - \frac{\mathcal{C}}{x}
\end{align*}
for some positive constant $\mathcal{C}$. This follows from the arguments of \cite[p.~1962]{Zhou} (see, in particular, line~7) with the parameters $f(n) = cn+ (a-c)$, $k=b$, $\delta = 1$, and $\ell =0$. Therefore, the contribution from the major arc exponentially dominates, and we obtain asymptotic equidistribution as claimed.
\end{proof}

\subsection*{Acknowledgements}
The author thanks W.~Craig, M.~Schlosser, and N.H.~Zhou, and the referees for many helpful comments on this note, in particular, for pointing out the related paper of Liu and Zhou~\cite{LZ}.

\pdfbookmark[1]{References}{ref}
\LastPageEnding

\end{document}